\documentclass[a4paper,11pt]{article}
\usepackage{amsmath}
\usepackage{amssymb}
\usepackage{mathrsfs}
\usepackage{amsfonts}
\usepackage{amsthm}
\usepackage{pstricks, pst-node, pst-text, pst-3d,psfrag}
\usepackage{graphicx}
\usepackage{subfigure}
\usepackage{indentfirst}
\usepackage{enumerate}
\usepackage{cite}
\usepackage{overpic}
\usepackage[colorlinks=true]{hyperref}
\hypersetup{urlcolor=blue, citecolor=red}

\theoremstyle{plain}
\newtheorem{thm}{Theorem}[section]

\newtheorem{lem}[thm]{Lemma}
\theoremstyle{definition}
\newtheorem{defn}[thm]{Definition}

\newtheorem{rmk}[thm]{Remark}

\numberwithin{equation}{section}

\makeatletter % `@' now normal "letter"
\@addtoreset{equation}{section}
\makeatother  % `@' is restored as "non-letter"

\begin{document}

\title{Generic Poincar\'{e}-Bendixson Theorem for singularly perturbed monotone systems with respect to cones of rank-$2$\thanks{Supported by NSF of China No.11825106, 11871231, 11771414 and 11971232, CAS Wu Wen-Tsun Key Laboratory of Mathematics, University of Science and Technology of China.}
}
\setlength{\baselineskip}{16pt}

\author {
 Lin Niu$^a$ and Xizhuang Xie$^{a,b}$  \thanks{
	Corresponding author.}
\\
\\[3pt]
{ $^{a}$School of Mathematical Sciences}\\
{ University of Science and Technology of China}\\
\vspace{\bigskipamount}
{ Hefei, Anhui, 230026, P. R. China}\\
{ $^{b}$School of Mathematical Sciences}\\
{ Huaqiao University}\\
{ Quanzhou, Fujian, 362021, P. R. China}\\
}

\date{}

    \maketitle
% insert the table of contents
%\tableofcontents

%---------------------SECTION DIVIDE LINE---------------------------
\begin{abstract}

We investigate the singularly perturbed monotone systems with respect to cones of rank $2$ and obtain the so called Generic Poincar\'{e}-Bendixson theorem for such perturbed systems, that is, for a bounded positively invariant set, there exists an open and dense subset $\mathcal{P}$ such that for each $z\in \mathcal{P}$, the $\omega$-limit set $\omega(z)$ that contains no equilibrium points is a closed orbit.

\par
\textbf{Keywords}:  Monotone systems with respect to $2$-cones, Regular perturbation, Singular perturbation, Generic Poincar\'{e}-Bendixson Theorem

\end{abstract}

\par \quad \quad \textbf{AMS Subject Classification (2020)}: 34C12, 34D15,
37C65

\section{Introduction}

In this paper, we are concerned with the dynamics of singularly perturbed monotone systems with respect to a cone $C$ of rank $2$. Roughly speaking, {\it a cone $C$ of rank $2$} is a closed subset of $\mathbb{R}^{n}$ that contains a linear subspace of dimension $2$ and no linear subspaces of higher dimension.
The concept of cones of rank $k$ was introduced by Fusco and Oliva \cite{FO91} in a finite-dimensional space,
and Kransosel'skii et al. \cite{KLS89} in a Banach space.

Monotone dynamical systems whose flows respect an order structure have been studied for a long time.
It was Hirsch \cite{H82,H85,H88,H90,H89,H91} who started the full description of the main dynamical properties in the setting of cooperative and competitive systems, whose flows preserve the partial order induced by a proper (i.e., convex, closed, solid, pointed) cone (see \cite{HS03,HS05,S95,S17}).
Among others, the celebrated result of a classical monotone flow with strong order preserving property is the so called Hirsch's Generic Convergence Theorem, which concludes that the set of all $x\in X$, for which the omega-limit set $\omega(x)$ belongs to the set of equilibria, is generic (open-dense, residual) in $X$ (see e.g., \cite{H85,S95,HS05}). Furthermore, for a classical smooth strongly monotone systems, precompact semi-orbits are generically convergent to equilibria in the continuous-time case \cite{Polacik88,Polacik89,SmithandThieme} or to cycles in the discrete-time case \cite{HessandPolacik,PolacikandTerescak,Terescak94,WangandYao}.

The classical monotone flow can be viewed as a monotone system with respect to a cone of rank-$1$.
And the long time behaviors of classical monotone flows are determined by a one dimensional space in the cone of rank-$1$.
For the smooth monotone flows with respect to cones of rank $k$ $(k\geq 2)$ in $\mathbb{R}^{n}$, S\'{a}nchez \cite{Sanchez09} first studied the structure of an omega limit set. By using cones of rank $2$, S\'{a}nchez \cite{Sanchez09} projected the dynamics into planes and proved a Poincar\'{e}-Bendixson type theorem: Any omega limit set of a pseudo-ordered orbit that contains no equilibrium is a closed orbit. Here, an orbit is called pseudo-ordered if it possesses one pair of distinct ordered points.
Besides, S\'{a}nchez's work is strongly related to the theory of R. A. Smith in \cite{RASmith79,RASmith80,RASmith87}, which obtained a Poincar\'{e}-Bendixson theorem for systems in a higher dimensional space.

Later, Feng, Wang and Wu \cite{FWW17} established the Poincar\'{e}-Bendixson theorem of pseudo-ordered orbits for strongly monotone semiflows (with respect to cones of rank $2$) on a Banach space without any smoothness assumption.
Very recently, Weiss and Margaliot \cite{WeissandMargaliot} described an important class of systems that are monotone with respect to a cone of rank $k$ and presented the Poincar\'{e}-Bendixson property for any bounded trajectory in the case $k=2$.
Feng, Wang and Wu \cite{FWW19} obtained the ``generic Poincar\'{e}-Bendixson Theorem" for strongly monotone flows with respect to cones of rank $2$, that is, for generic (belonging to an open and dense set) points, the $\omega$-limit set containing no equilibrium is a single closed orbit.

It deserves to point out that $C^1$-Closing Lemma was used as an effective tool in perturbation theory both for Hirsch's work \cite{H85} and S\'{a}nchez's dynamical result \cite{Sanchez09}.
As a consequence, the perturbation of a monotone system began to attract considerable attention.
Hirsch \cite{H85} obtained that the regular perturbation of a cooperative irreducible vector field is at most eventually cooperative rather than cooperative.
While, for singular perturbations of a cooperative system, Sontag and Wang \cite{WS08} showed that such perturbed system is also eventually cooperative.
Later, Niu \cite{Niu19} showed the $C^1$-regular perturbation of a competitive irreducible vector field is at most eventually competitive.
For the high-rank cones, S\'{a}nchez \cite{Sanchez09} considered the regular perturbation in the sense that a sequence $F_n$ of $C^1$-vector field converges to $F$ in the $C^1$-topology.

Based on these results, one can see that the perturbed systems (eventual cases) inherit sort of dynamical properties of unperturbed cooperative or competitive systems (see \cite{H85,NW18}).
Thus, it is interesting to consider the (both regular and singular) perturbations of a monotone system with respect to high-rank cones and investigate the dynamics of the perturbed systems .

In the present paper, we try to carry out this task by focusing on a singularly perturbed monotone system (with respect to cones of rank $2$) having the form:
\begin{equation}\label{introduction-system}
\begin{cases}
\frac{dx}{dt} & = f_{0}(x,y,\epsilon),\\
\epsilon \frac{dy}{dt} & =g_{0}(x,y,\epsilon),
\end{cases}
\end{equation}
with a positive parameter $\epsilon$ near zero and $(x,y)\in  U \times V$, where $U \subset \mathbb{R}^{n}$, $V \subset \mathbb{R}^{m}$ are open and bounded sets such that $ U \times V$ contains a compact set $D_{\epsilon}$. Under a series of hypotheses (A1)-(A6) (see the details in Section $2$), including that the limiting system ($\epsilon=0$) is monotone with respect to a cone of rank $2$ in $\mathbb{R}^{n}$, we are able to deduce the dynamics of perturbed system from the unperturbed monotone system ($\epsilon =0$) with respect to a cone of rank $2$. More precisely, the Poincar\'{e}-Bendixson type theorem is inherited for the perturbed system:

\vskip 3mm

\textbf{Theorem A} {\it For singular perturbed system \eqref{introduction-system} satisfying (A1)-(A6), there exists an open and dense subset $\mathcal{P}_{\epsilon} \subset D_{\epsilon}$ such that for each $z\in \mathcal{P}_{\epsilon}$, the $\omega$-limit set $\omega(z)$ containing no equilibrium is a closed orbit.}

\vskip 3mm

We call this theorem as ``generic Poincar\'{e}-Bendixson Theorem" for singular perturbed monotone system with respect to cones of rank $2$.
By the geometric singular perturbation theory (see \cite{Fenichel79,Jones95,Nipp92}) under the assumptions (A1)-(A6), there is a slow manifold which attracts all flows in $D_{\epsilon}$. The flow on the slow manifold can be treated as a $C^{1}$-regular perturbation of the limiting flow for $\epsilon =0$. Further, all the dynamics in $D_{\epsilon}$ could be tracked by the flows on the slow manifold.

As one may know, one of the typical examples of a monotone flow with respect to a cone of rank $2$ is the three-dimensional competitive systems.
If the limiting flow is competitive, system \eqref{introduction-system} can be treated as a singular perturbed competitive system, and the flow on the slow manifold is eventually competitive (see \cite{Niu19}). For an eventually competitive system, the asymptotic behaviors of the complete (full) orbits have been investigated,
due to a priori drawback restriction of the so called Non-oscillation Principle for the full orbits (see Niu and Wang \cite{NW18}).
To the best of our knowledge, very limited work has been done on the dynamics (more specifically, the structure of $\omega$-limit sets) of the singular perturbed competitive system.
However, in the view of cones of rank $2$,  the asymptotic behaviors of the positive orbits can be investigated in this paper.
By applying Theorem A, we can obtain the Generic Poincar\'{e}-Bendixson theorem for a singularly perturbed three-dimensional competitive systems. We will present a concrete example to illustrate our main result of generic Poincar\'{e}-Bendixson Theorem.

This paper is organized as follows. In section $2$, we introduce some basic definitions and assumptions for singularly perturbed monotone systems with respect to cones of rank $2$, and present the main result (see Theorem \ref{thm}). Section $3$ is devoted to give the detailed proof of the main theorem.
Finally, in Section $4$, an example for singular perturbations is presented and numerical examples are given to illustrate the effectiveness of our theoretical results.

%\textit{\textbf{Keywords:}} xxxx
%------------------------SECTION DIVIDE LINE-------------------
\section{Basic definitions and the main theorem}

A nonempty closed set $C\subset \mathbb{R}^{n}$ is called {\it a cone of rank $k$} ($k$-cone) if it satisfies $\alpha C\subset C$ for all $\alpha \in \mathbb{R}$ and $\max\{ \text{dim} W : C\supset W \ \ \text{linear} \ \ \text{subspace}\}=k$. A cone $C$ is {\it solid} if $\text{Int }C \neq \emptyset$; and $C$ is called {\it $k$-solid} if there is a $k$-dimensional linear subspace $W$ such that $W\backslash \{0\}\subset \text{Int }C$.
We write
\begin{alignat*}{2}
x &\thicksim y \quad& \text{if}&\ \ y - x \in C, \\
x &\thickapprox y & \text{if}&\ \ y-x \in \text{Int }C.
\end{alignat*}

A flow $\phi_t$ is called {\it monotone with respect to a $k$-solid cone $C$} if $\phi_t(x)\thicksim \phi_t(y)$ whenever $x\thicksim y$ and $t\geq 0$. And $\phi_t$ is called {\it strongly monotone with respect to $C$} if $\phi_t$ is monotone with respect to $C$ and $\phi_t(x)\thickapprox \phi_t(y)$ whenever $x\neq y$, $x\thicksim y$ and $t> 0$.

Consider the system:
\begin{equation}\label{ODEequation}
\frac{dx}{dt} = F(x),
\end{equation}
for which $F: O \rightarrow \mathbb{R}^{n}$ is a $C^{1}$-vector field and $O \subset \mathbb{R}^{n}$ is an open convex set. We denote by $\phi_t$ the flow generated by \eqref{ODEequation}.

Let $U_{F}^{pq}(t)$ be the solution of
\begin{equation}\label{difference-equation}
\begin{cases}
\frac{d}{dt}U = A_{F}^{pq}(t) U \\
U(0)=I,
\end{cases}
\end{equation}
where $A_{F}^{pq}(t)=\int^{1}_{0}DF(s\phi_t(p)+(1-s)\phi_t(q))ds$ for any $p,q\in O$.

\begin{defn}\label{C-cooperativeness}
The system \eqref{ODEequation} is called $C$-cooperative if for any $p,q\in O$, the matrix $U_{F}^{pq}(t)$ satisfies $U_{F}^{pq}(t)[C\backslash \{0\}]\subset \text{Int }C$ for all $t>0$.
\end{defn}

\begin{rmk}
By S\'{a}nchez \cite[Proposition 1]{Sanchez09}, if system \eqref{ODEequation} is $C$-cooperative, then the flow of \eqref{ODEequation} is strongly monotone with respect to cone $C$.
\end{rmk}

In this paper, we are concerned with the study of the singular perturbed differential equations of the form

\begin{equation}\label{slowequ}
\begin{cases}
\frac{dx}{dt} = f_{0}(x,y,\epsilon),\\
\epsilon \frac{dy}{dt} =g_{0}(x,y,\epsilon).
\end{cases}
\end{equation}
System \eqref{slowequ} can be reformulated with a change of time scale as
\begin{equation}\label{Fastequ}
\begin{cases}
\frac{dx}{d\tau} =\epsilon f_{0}(x,y,\epsilon),\\
\frac{dy}{d\tau} =g_{0}(x,y,\epsilon),
\end{cases}
\end{equation}
where $\tau=t/\epsilon $. The time scale given by $\tau$ is said to be fast whereas that for $t$ is slow. When $\epsilon$ is small enough, we call \eqref{slowequ} the slow system and \eqref{Fastequ} the fast system. And the two systems are equivalent as long as $\epsilon \neq 0$.

Let $C^{r}_{b}$ denote a class of functions such that if a function $f$ is in $C^{r}$ and its derivatives up to order $r$ as well as $f$ are bounded, then $f\in C^{r}_{b}$. Throughout this paper, let $Df(x_0)$ denote the derivatives of $f$ evaluated at $x_0$ with respect to the variable $x$. Meanwhile, $D_{x}f(x_0,y_0)$ and $D_{y}f(x_0,y_0)$ denote the partial derivatives of $f$ with respect to $x$ and $y$ evaluated at $(x_0,y_0)$, respectively.
A family of sets $D_{\epsilon} \subset \mathbb{R}^{n}$ is upper semicontinuous at $\epsilon_1 \in [0,\epsilon_{0}]$ if given a neighborhood $U$ of $D_{\epsilon_1}$, there exists a $\delta>0$ such that $D_{\epsilon}\subset U$ for all $\epsilon \in (\epsilon_1-\delta, \epsilon_1+\delta)$. $D_{\epsilon}$ is lower semicontinuous at $\epsilon_1 \in [0,\epsilon_{0}]$ if given any open set $N$ such that $N\cap D_{\epsilon_1}$ is non-empty, there exists a $\delta>0$ such that $N\cap D_{\epsilon}$ is non-empty for all  $\epsilon \in (\epsilon_1-\delta, \epsilon_1+\delta)$. It is continuous at $\epsilon$ if it is both lower and upper semicontinuous at $\epsilon$ and continuous if it is continuous at every $\epsilon \in [0,\epsilon_{0}]$.
Then, we have the following assumptions, where the integer $r>1$ and the positive number $\epsilon_{0}$ are fixed from now on.

\begin{enumerate}[({A}1)]
\item

Let $U \subset \mathbb{R}^{n}$ and $V \subset \mathbb{R}^{m}$ be open and bounded sets. The functions
\begin{equation*}
f_{0} : U \times V \times [0,\epsilon_{0}] \rightarrow \mathbb{R}^{n}
\end{equation*}
and
\begin{equation*}
g_{0} : U \times V \times [0,\epsilon_{0}] \rightarrow \mathbb{R}^{m}
\end{equation*}
are both of class $C^{r}_{b}$.

\item
There is a function
\begin{equation*}
h_0 : U \rightarrow V
\end{equation*}
in $C^{r}_{b}$ such that $g_{0}(x,h_0(x),0)=0$ for all $x$ in $U$.

\item
All eigenvalues of the matrix $D_{y}g_{0}(x,h_0(x),0)$ have negative real parts for every $x\in U$.

\item
There exists a family of convex compact sets $D_{\epsilon} \subset U \times V$, which depend continuously on $\epsilon \in [0,\epsilon_{0}]$, such that \eqref{slowequ}
is positively invariant on $D_{\epsilon}$ for $\epsilon \in (0,\epsilon_{0}]$.

\item
For each $x\in U$, the system
\begin{equation}\label{degenerate}
\frac{dz}{d\tau} =g_{1}(x,z,0) \triangleq g_{0}(x,z+h_0(x),0)
\end{equation}
is defined on $\{z\in \mathbb{R}^{m} : z + h_0(x) \in V \}$. And
the steady state $z = 0$ of \eqref{degenerate} is globally asymptotically stable on $ \{z : z + h_0(x) \in V \}$.

\item

Let $K_{0}$ be the projection of $D_{0} \cap \{(x,y): y = h_0(x),x \in U\}$ onto the $x$-axis.
For $x\in K_{0}$, the system
\begin{equation}\label{limitingequ}
\frac{dx}{dt} =f_{0}(x,h_0(x),0)
\end{equation}
is $C$-cooperative, where $C\subset \mathbb{R}^{n}$ is a cone of rank $2$.

\end{enumerate}

Now, we present the {\it generic Poincar\'{e}-Bendixson Theorem}:

\begin{thm}\label{thm}
Assume that (A1)-(A6) hold. Then there exists a positive constant $\epsilon^{*} < \epsilon_{0}$ such that for each $\epsilon \in (0,\epsilon^{*})$, system \eqref{slowequ} has the following property: there exists an open and dense subset $\mathcal{P}_{\epsilon} \subset  D_{\epsilon}$ such that, for each $z\in \mathcal{P}_{\epsilon}$, $\omega$-limit set $\omega(z)$ that contains no equilibrium is a single closed orbit.
\end{thm}

\section{Details of the proof}

Our approach to prove the main theorem consists of two steps. First, we focus on the regular perturbations of system \eqref{limitingequ}. And then, we utilize the geometric construction (see, e.g., \cite{Fenichel79,Saka90}) of system \eqref{slowequ} to show the generic dynamics.

\subsection{Regularly perturbed monotone systems with respect to high-rank cone}
For a general case, we consider the regular perturbations of system \eqref{ODEequation} in this subsection.
We focus on the system of ODE's
\begin{equation}\label{perturbed-equation}
\frac{dx}{dt} = G(x),
\end{equation}
for which $G: O \rightarrow \mathbb{R}^{n}$ is a $C^{1}$-vector field. Let $W\subset O$ be a convex compact subset. %and the system \eqref{ODEequation} is $C$-cooperative.
Then we obtain the following lemma.

\begin{lem}\label{regularlem}
Let system \eqref{ODEequation} be a $C$-cooperative system. Then there exists $\delta >0$ with the following property: If $\parallel F(z)-G(z)\parallel+\parallel DF(z)-DG(z)\parallel<\delta$ for all $z\in O$, and $W$ is positively invariant under the flow $\psi_{t}$ generated by $G$, then there exists $t_{*}>0$ such that
the matrix $U_{G}^{pq}(t)$ satisfies $U_{G}^{pq}(t)[C\backslash \{0\}]\subset \text{Int }C$ for any $p,q\in W$ and $t\geq t_*$.
\end{lem}

\begin{proof}
Pick $t_*>0$, we consider the matrices $$A_{G}^{pq}(t)=\int^{1}_{0}DG(s\psi_t(p)+(1-s)\psi_t(q))ds$$ and the solution $U_{G}^{pq}(t)$ of
\begin{equation}\label{perturbed-difference-equation}
\begin{cases}
\frac{d}{dt}U = A_{G}^{pq}(t) U \\
U(0)=I,
\end{cases}
\end{equation}
for any $p,q\in W$.

We first prove the positiveness of $U_{G}^{pq}(t)$ for $t\in [t_*, 2t_*]$, i.e., $U_{G}^{pq}(t)[C\backslash \{0\}]\subset \text{Int }C$ for $t\in [t_*, 2t_*]$. Since system \eqref{ODEequation} is $C$-cooperative, there exists a $\delta_{1}>0$ such that $B_{\delta_{1}}(U_{F}^{pq}(t)v)\subset \text{Int }C$ for all $v\in C\backslash \{0\}$ with $|v|=1$. A positive $\delta$ can be found such that if $\parallel F(z)-G(z)\parallel +\parallel DF(z)-DG(z)\parallel<\delta$ then $\parallel U_{F}^{pq}(t)-U_{G}^{pq}(t) \parallel \leq \frac{\delta_{1}}{2}$ hold for $t\in [t_{*},2t_{*}]$.

In fact, by the definition of flow, we have
\begin{equation*}
\begin{cases}
\frac{d}{dt}\phi_{t}(z) = F(\phi_{t}(z)),\\
\frac{d}{dt}\psi_{t}(z) = G(\psi_{t}(z)).
\end{cases}
\end{equation*}
Then,
\begin{align*}
\parallel \phi_{t}(z)-\psi_{t}(z)\parallel & \leq \int^{t}_{0}\parallel F(\phi_{s}(z))-G(\psi_{s}(z))\parallel ds \\ & \leq \int^{t}_{0} \parallel F(\phi_{s}(z))-F(\psi_{s}(z))\parallel ds + \int^{t}_{0} \parallel F(\psi_{s}(z))-G(\psi_{s}(z)) \parallel ds.
\end{align*}
Since $F$ and $G$ are $C^{1}$ and $t_{*} \leq t \leq 2t_{*}$, there exists some $M > 0$ such that $\parallel F(\phi_{s}(z))-F(\psi_{s}(z))\parallel \leq M \parallel \phi_{s}(z)-\psi_{s}(z)\parallel$. Then $\parallel \phi_{t}(z)-\psi_{t}(z)\parallel \leq \int^{t}_{0}M \parallel \phi_{s}(z)-\psi_{s}(z)\parallel ds + \delta t$.

Using the Gronwall's inequality, we obtain
\begin{equation*}
\parallel \phi_{t}(z)-\psi_{t}(z)\parallel \leq \delta 2t_{*} e^{M2t_{*}}.
\end{equation*}
Hence, for $s\in [0,1]$,
\begin{align*}
&\ \ \ \ \parallel [s\phi_{t}(p)+(1-s)\phi_{t}(q)]-[s\psi_{t}(p)+(1-s)\psi_{t}(q)]\parallel
%\\ = & \parallel s[\phi_{t}(p)-\psi_{t}(p)]+(1-s)[\phi_{t}(q)-\psi_{t}(q)]\parallel
\\ & \leq s\parallel \phi_{t}(p)-\psi_{t}(p)\parallel+ (1-s)\parallel \phi_{t}(q)-\psi_{t}(q)\parallel\\ & \leq s\delta 2t_{*} e^{M2t_{*}}+ (1-s)\delta 2t_{*} e^{M2t_{*}} =  \delta 2t_{*} e^{M2t_{*}}.
\end{align*}
Note that $U_{F}^{pq}(t)$ and $U_{G}^{pq}(t)$ are the solutions of \eqref{difference-equation} and \eqref{perturbed-difference-equation}, respectively, %$U_{F}^{pq}(t)-I=\int ^{t}_0 A_{F}^{pq}(s)U_{F}^{pq}(s)ds$ and $U_{G}^{pq}(t)-I=\int ^{t}_0 A_{G}^{pq}(s)U_{G}^{pq}(s)ds$.
then
\begin{align*}
\parallel U_{F}^{pq}(t)-U_{G}^{pq}(t) \parallel \leq & \int^{t}_{0} \parallel  A_{F}^{pq}(s)U_{F}^{pq}(s)-A_{G}^{pq}(s)U_{G}^{pq}(s) \parallel ds
%\\ \leq & \int^{t}_{0} \parallel A_{F}^{pq}(s) [U_{F}^{pq}(s)-U_{G}^{pq}(s)]+[A_{F}^{pq}(s)-A_{G}^{pq}(s)]U_{G}^{pq}(s) \parallel ds
\\ \leq & \int^{t}_{0} \parallel A_{F}^{pq}(s) \parallel \parallel U_{F}^{pq}(s)-U_{G}^{pq}(s)\parallel ds \\ & +\int^{t}_{0} \parallel A_{F}^{pq}(s)-A_{G}^{pq}(s) \parallel \parallel U_{G}^{pq}(s) \parallel ds.
\end{align*}
Using the Gronwall's inequality again, it follows that
\begin{equation*}
\parallel U_{F}^{pq}(t)-U_{G}^{pq}(t) \parallel \leq e^{\int^{t}_{0} \parallel A_{F}^{pq}(s) \parallel ds} \int^{t}_{0} \parallel A_{F}^{pq}(s)-A_{G}^{pq}(s) \parallel \parallel U_{G}^{pq}(s) \parallel ds.
\end{equation*}
There exist $N>0$ and $L>0$ such that $\parallel U_{G}^{pq}(s) \parallel \leq N$ and $\parallel A_{F}^{pq}(s) \parallel \leq L$ for all $t\in [t_*, 2t_*]$.
Meanwhile,
\begin{align*}
\parallel A_{F}^{pq}(t)-A_{G}^{pq}(t)\parallel \leq & \int^{1}_{0}\parallel DF(s\phi_t(p)+(1-s)\phi_t(q))-DG(s\psi_t(p)+(1-s)\psi_t(q))\parallel ds \\ \leq & \int^{1}_{0}\parallel DF(s\phi_{t}(p)+(1-s)\phi_{t}(q))-DF(s\psi_{t}(p)+(1-s)\psi_{t}(q))\parallel \\ & +\parallel DF(s\psi_{t}(p)+(1-s)\psi_{t}(q))-DG(s\psi_{t}(p)+(1-s)\psi_{t}(q))\parallel ds\\ \leq & \int^{1}_{0} \parallel DF(s\phi_{t}(p)+(1-s)\phi_{t}(q))-DF(s\psi_{t}(p)+(1-s)\psi_{t}(q))\parallel ds+ \delta.
\end{align*}

Since $F$ is $C^{1}$ and $\parallel [s\phi_{t}(p)+(1-s)\phi_{t}(q)]-[s\psi_{t}(p)+(1-s)\psi_{t}(q)]\parallel \leq \delta 2t_{*} e^{M2t_{*}}$, one can choose $\delta < \frac{1}{2}\cdot \frac{\delta_{1}}{2}\cdot \frac{1}{2Nt_{*}e^{2t_{*}L}}$ small enough such that $\parallel DF(s\phi_{t}(p)+(1-s)\phi_{t}(q))-DF(s\psi_{t}(p)+(1-s)\psi_{t}(q))\parallel < \frac{1}{2}\cdot \frac{\delta_{1}}{2}\cdot \frac{1}{2Nt_{*}e^{2t_{*}L}}$. So, $\parallel A_{F}^{pq}(t)-A_{G}^{pq}(t)\parallel \leq \frac{\delta_{1}}{2}\cdot \frac{1}{2Nt_{*}e^{2t_{*}L}}$, which implies that $\parallel U_{F}^{pq}(t)-U_{G}^{pq}(t) \parallel \leq \frac{\delta_{1}}{2}$.
Thus, we obtain that $U_{G}^{pq}(t)[C\backslash \{0\}]\subset \text{Int }C$ for $t\in [t_{*},2t_{*}]$ and $p,q\in W$.

For $t>2t_*$, let us write $t=t_0+kt_*$ with $t_0\in [t_*,2t_*)$. Define $t_j=jt_*$, $p_j=\psi_{t_j}(p)$, $q_j=\psi_{t_j}(q)$ with $j=1,2,\cdots,k$. It is clear that $p_j, q_j\in W$ if $W$ is positively invariant. Since $\psi_{t}(p)-\psi_{t}(q)=U_{G}^{pq}(t)(p-q)$, we obtain that
\begin{equation*}
U_{G}^{pq}(t)=U_{G}^{p_kq_k}(t_0)U_{G}^{p_{k-1}q_{k-1}}(t_*)\cdots U_{G}^{p_1q_1}(t_*)U_{G}^{pq}(t_*).
\end{equation*}
By the preceding proof, $U_{G}^{pq}(t)[C\backslash \{0\}]\subset \text{Int }C$ for $t>2t_*$.

Thus, we have proved that the system \eqref{perturbed-equation} is $C$-cooperative for $t\geq t_*$.
\end{proof}

\subsection{Proof of the main theorem}

Before proving Theorem \ref{thm}, we give the following lemma about the singularly perturbed systems on $\mathbb{R}^{n} \times \mathbb{R}^{m} \times [0,\epsilon_{0}]$, which is a restatement of the Theorem 2.1 and Theorem 3.1 in Sakamoto \cite{Saka90}.

\begin{lem}\label{singularlemm}
Consider the system
\begin{equation}\label{singularequ}{}
\begin{cases}
\frac{dx}{d\tau}  = \epsilon f(x,y,\epsilon),\\
\frac{dy}{d\tau}  = g(x,y,\epsilon),
\end{cases}
\end{equation}
where $f :\mathbb{R}^{n} \times \mathbb{R}^{m} \times [0,\epsilon_{0}] \rightarrow \mathbb{R}^{n}$ is $C^{r}_{b}$ and
$g :\mathbb{R}^{n} \times \mathbb{R}^{m} \times [0,\epsilon_{0}] \rightarrow \mathbb{R}^{m}$ is $C^{r}_{b}$. For $x\in \mathbb{R}^{n}$, there is a $C^{r}_{b}$ function $\overline{h} :\mathbb{R}^{n} \rightarrow \mathbb{R}^{m}$ such that $g(x,\overline{h}(x),0)=0$. And there exists a positive constant $\mu$ such that all eigenvalues of the matrix $D_{y}g(x,\overline{h}(x),0)$ have negative real parts less than $-\mu$ for every $x \in \mathbb{R}^{n}$.

Then, there exists a positive number $\epsilon_{1} < \epsilon_{0}$ such that for every $\epsilon \in (0,\epsilon_{1}]$:
\begin{enumerate}[$1)$]
\item
There exists a $C_{b}^{r-1}$ function $h : \mathbb{R}^{n} \times [0,\epsilon_{1}] \rightarrow \mathbb{R}^{m}$
such that the set $\mathbb{C}_{\epsilon}$ defined by
\begin{equation*}
\mathbb{C}_{\epsilon} = \{ (x,h(x,\epsilon)) : x \in \mathbb{R}^{n} \},\  \epsilon \in (0,\epsilon_{1}]
\end{equation*}
is invariant under the flow generated by \eqref{singularequ} and
\begin{equation*}
\underset{x \in \mathbb{R}^{n}}{\sup} \{| h(x,\epsilon) - \overline {h}(x)| : x \in \mathbb{R}^{n} \} \sim O(\epsilon), as \ \epsilon \rightarrow 0.
\end{equation*}
In particular, we have $h(x,0) = \overline {h}(x)$ for all $x \in \mathbb{R}^{n}$.

\item
There is an $(n + m)$-dimensional $C^{r-1}$ submanifold $W^{s}(\mathbb{C}_{\epsilon})$. It is characterized by
\begin{equation*}
W^{s}(\mathbb{C}_{\epsilon}) = \{ (x_{0},y_{0}) : \underset {\tau \geq 0}{\sup} | y(\tau;x_{0},y_{0})-h_{\epsilon}(x(\tau;x_{0},y_{0})) | e^{\frac{\mu \tau}{4}} < \infty \},
\end{equation*}
where $(x(\tau;x_{0},y_{0}),y(\tau;x_{0},y_{0}))$ is the solution of \eqref{singularequ} passing through $(x_{0},y_{0})$ and $h_{\epsilon}(x)=h(x,\epsilon)$ is the function defining $\mathbb{C}_{\epsilon}$.

\item
The manifold $W^{s}(\mathbb{C}_{\epsilon})$ is a disjoint union of the $m$-dimensional $C^{r-1}$ manifold $W^{s}(\xi)$:

\begin{equation*}
W^{s}(\mathbb{C}_{\epsilon}) = \underset{\xi \in \mathbb{R}^{n}}{\cup} W^{s}(\xi).
\end{equation*}

Moreover, $W^{s}(\xi)$ is characterized as
\begin{equation*}
W^{s}(\xi)= \{ (x_{0},y_{0}) : \underset{\tau \geq 0}{\sup} | \widetilde{x}(\tau) | e^{\frac{\mu\tau}{4}} < \infty,\ \underset{\tau \geq 0}{\sup} | \widetilde{y}(\tau) | e^{\frac{\mu\tau}{4}} < \infty \},
\end{equation*}
where $\widetilde{x}(\tau)=x(\tau;x_{0},y_{0})- H_{\epsilon}(\xi)(\tau)$, $\widetilde{y}(\tau)=y(\tau;x_{0},y_{0})- h_{\epsilon}(H_{\epsilon}(\xi)(\tau))$, where $H_{\epsilon}(\xi)(\tau)$ stands for a unique solution of $\frac{dx}{d\tau}=\epsilon f(x,h_{\epsilon}(x),\epsilon)$, $x(0)= \xi \in \mathbb{R}^{n}$.

\item
There is a constant $\delta_{0} > 0$ such that if a solution $(x(\tau),y(\tau))$ of \eqref{singularequ} satisfies
\begin{equation*}
\underset {\tau \geq 0}{\sup} | y(\tau) - h_{\epsilon}(x(\tau)) | < \delta_{0},
\end{equation*}
then $(x(0),y(0)) \in W^{s}(\mathbb{C}_{\epsilon})$.

\item
The fibers are positively invariant in the sense that
\begin{equation*}
W^{s}(H_{\epsilon}(\xi)(\tau)) = \{ (x(\tau;,x_{0},y_{0}),y(\tau;x_{0},y_{0})) : (x_{0},y_{0}) \in W^{s}(\xi) \},
\end{equation*}
for each $\tau \geq 0$.

\item
The fibers restricted to the $\delta_{0}$ neighborhood of $\mathbb{C}_{\epsilon}$, denoted by $W^{s}_{\epsilon, \delta_{0}}$, can be parameterized as follows. Let $L_{\delta_{0}}=\{ z\in \mathbb{R}^{m} : |z|\leq \delta_{0} \}$. There are two $C^{r-1}_{b}$ functions
\begin{equation*}
P_{\epsilon, \delta_{0}} : \mathbb{R}^{n} \times L_{\delta_{0}} \rightarrow \mathbb{R}^{n},
\end{equation*}
\begin{equation*}
Q_{\epsilon, \delta_{0}} : \mathbb{R}^{n} \times L_{\delta_{0}} \rightarrow \mathbb{R}^{m},
\end{equation*}
and a map
\begin{equation*}
T_{\epsilon, \delta_{0}} : \mathbb{R}^{n} \times L_{\delta_{0}} \rightarrow \mathbb{R}^{n} \times \mathbb{R}^{m}
\end{equation*}
mapping $(\xi, \eta)$ to $(x,y)$, where
\begin{equation*}
x = \xi + P_{\epsilon, \delta_{0}}(\xi, \eta), \ y = h(x, \epsilon) + Q_{\epsilon, \delta_{0}}(\xi, \eta),
\end{equation*}
such that
\begin{equation*}
W^{s}_{\epsilon, \delta_{0}}(\xi) = T_{\epsilon, \delta_{0}}(\xi, L_{\delta_{0}}).
\end{equation*}
\end{enumerate}
\end{lem}

\begin{rmk}\label{singularrmk}

\begin{enumerate}[(1)]
\item
The $\delta_{0}$ in property 4) of Lemma \ref{singularlemm} can be chosen uniformly for $\epsilon \in (0,\epsilon_{0}]$.

\item
The property 3) of Lemma \ref{singularlemm} is often referred to as the asymptotic phase property in the way
\begin{equation*}
| x(\tau;x_{0},y_{0}) - H_{\epsilon}(\xi)(\tau) | \rightarrow 0,
\end{equation*}
\begin{equation*}
| y(\tau;x_{0},y_{0}) - h_{\epsilon}(H_{\epsilon}(\xi)(\tau) | \rightarrow 0,
\end{equation*}
as $\tau \rightarrow \infty$.
\end{enumerate}
\end{rmk}

\vskip 4mm

In order to use Sakamoto's results in Lemma \ref{singularlemm}, we firstly extend the vector fields from $U \times V $ to $\mathbb{R}^{n} \times \mathbb{R}^{m}$ for $\epsilon \in [0,\epsilon_{0}]$. The technique is standard by \cite{Nipp92}, which can also be found in \cite{WS08}, such that the extended system:
\begin{equation*}\label{fastequation}
\begin{cases}
\frac{dx}{d\tau}  = \epsilon f(x,y,\epsilon),\\
\frac{dy}{d\tau}  = g(x,y,\epsilon),
\end{cases}
\end{equation*}
satisfies the assumptions $(A1)$-$(A6)$ and the assumptions for the geometric singular perturbation in Lemma \ref{singularlemm}. Moreover, $\overline{h}(x)$ coincides with $h_{0}(x)$ on $K$, $f$ and $g$ coincide with $f_{0}$, $g_{0}$ on $\Omega_{d_{1}}$, respectively. Where $K$ is a compact set with $K_0\subset K\subset U$, $\Omega_{d_{1}} \triangleq \{ (x,y) : x \in K, y\in V, |y-h_{0}(x)| \leq d_{1} \}$ and $d_{1}>0$ is fixed such that $\delta_{0}$ in Lemma \ref{singularlemm} is less than $d_{1}$.

\vskip 4mm

\begin{proof}[Proof of Theorem \ref{thm}]
First, we consider the solutions on the invariant manifold $\mathbb{C}_{\epsilon}$ satisfying
\begin{equation}\label{perturbedequation}
\begin{cases}
\frac{dx}{dt}  =f(x,h_{\epsilon}(x),\epsilon),\\
y(t) =h_{\epsilon}(x(t)).
\end{cases}
\end{equation}
For brevity, we only focus on the $x$-direction on the invariant manifold $\mathbb{C}_{\epsilon}$, since $y=h_{\epsilon}(x)$.
Clearly, the limiting equation of \eqref{perturbedequation} is \eqref{limitingequ} as $\epsilon$ approaches zero.
For system \eqref{limitingequ}, the matrix $U_{f_{0}}^{pq}(t)$ satisfies $U_{f_{0}}^{pq}(t)[C\backslash \{0\}]\subset \text{Int }C$ for all $p,q\in K_{0}$ and $t>0$. By the continuity of $D_{\epsilon}$ and $h_{\epsilon}(x)$ at $\epsilon=0$, we can pick an $\epsilon_{2}<\epsilon_{1}$ small enough such that $U_{f_{0}}^{pq}(t)$ satisfies $U_{f_{0}}^{pq}(t)[C\backslash \{0\}]\subset \text{Int }C$ for $p,q\in K_{\epsilon}$ and $\epsilon \in (0,\epsilon_{2})$, where $K_{\epsilon}$ is the projection of $\mathbb{C}_{\epsilon} \cap D_{\epsilon}$ to the $x$-axis. Note also that $D_{\epsilon}$ is positively invariant under \eqref{perturbedequation} and $\mathbb{C}_{\epsilon}$ is an invariant manifold. Then, $K_{\epsilon}$ is positively invariant under the flow $\psi^{\epsilon}_{t}$ of \eqref{perturbedequation}. Applying Lemma \ref{regularlem}, we obtain that there exist an $\epsilon_3 \leq \epsilon_2$ and some $t_{*}>0$ such that for each $\epsilon \in (0,\epsilon_3)$, system \eqref{perturbedequation} is $C$-cooperative for $t\geq t_*$.

By \cite[Proposition 1]{Sanchez09}, it is clear that the flow of system \eqref{perturbedequation} satisfies the assumption (FWW) (see \cite[p.4]{FWW19}) for $t\geq t_*$. Then \cite[Theorem 5.3]{FWW19} implies that there exists an open and dense subset $M_{\epsilon}\subset K_{\epsilon}$ such that for any $x\in M_{\epsilon}$, the $\omega$-limit set $\omega(x)$ containing no equilibrium is a single closed orbit.

It follows from the Lemma $7$ in \cite{WS08} that there exist $\epsilon_4>0$ and $d\in (0,\delta_{0})$ such that, for each $\epsilon \in (0,\epsilon_4)$, $(x_0,y_0) \in W^{s}(\mathbb{C}_{\epsilon})$ whenever $(x_0,y_0)\in D_{\epsilon}$ satisfies $|y_{0} - h_{\epsilon}(x_{0})| < d$. Moreover, Lemma \ref{singularlemm} $3)$ guarantees that $(x_0,y_0)\in W^{s}_{\epsilon, d}(\xi)$, where $\xi \in K_{\epsilon}$ and $W^{s}_{\epsilon, d}(\xi)$ is the stable fiber restricted to the $d$ neighborhood of $\mathbb{C}_{\epsilon}$.

Let $N_{\epsilon}=\{(x,y)\in D_{\epsilon} : |y - h_{\epsilon}(x)| < d \}$.
In the following, we will show that the set $\underset{\xi \in M_{\epsilon}}{\cup} W^{s}_{\epsilon, d}(\xi)$ is open and dense in $N_{\epsilon}$.
By Lemma \ref{singularlemm} $6)$, the stable fiber can be characterized as $W^{s}_{\epsilon, d}(\xi) = T_{\epsilon, d}(\xi, L_d)$, where $L_d=\{ z\in \mathbb{R}^{m} : |z|<d \}$. In other words, $\underset{\xi \in M_{\epsilon}}{\cup} W^{s}_{\epsilon, d}(\xi)=T_{\epsilon, d}(M_{\epsilon},L_d)$. For any $(x_0,y_0)\in N_{\epsilon}$, there exists $(\xi_0,\eta_0)\in K_{\epsilon} \times L_d$ such that $(x_0,y_0)=T_{\epsilon, d}(\xi_0, \eta_{0})$. Since $M_{\epsilon}$ is dense in $K_{\epsilon}$, there exists a sequence $\{ \xi_n\}\subset M_{\epsilon}$ such that $\xi_n \to \xi_0$. Let $\eta_n= \eta_0$ and $(x_n,y_n)=T_{\epsilon, d}(\xi_n,\eta_n)$ for $n=1,2,\cdots$, then $(\xi_n,\eta_n)\to (\xi_0,\eta_0)$ and $(x_n,y_n)\in T_{\epsilon, d}(M_{\epsilon},L_d)$. Since the mapping $T_{\epsilon, d}$ is continuous, one has $(x_n,y_n)\to (x_0,y_0)$ as $n\to \infty$. Thus, the set $T_{\epsilon, d}(M_{\epsilon},L_d)$ is dense in $N_{\epsilon}$. In order to prove the openess of $T_{\epsilon, d}(M_{\epsilon},L_d)$, we only need to show the inverse of the mapping $T_{\epsilon, d}$ is continuous, because  both $M_{\epsilon}$ and $L_d$ are open sets. Clearly, the mapping $T_{\epsilon, d}$ is both injective and surjective. And by the proof of the Claim $3.7$ in \cite{Saka90}, the inverse of $T_{\epsilon, d}$ is continuous (In fact, there exists a mapping $\phi^*$, see \cite[Lemma 3.9]{Saka90}, such that $\xi_0=x_0-\phi^*(0)$ and $\eta_0=y_0-h_{\epsilon}(x_{0})$ hold for $(x_0,y_0)\in N_{\epsilon}$). Thus, the set $T_{\epsilon, d}(M_{\epsilon},L_d)$ is open and dense in $N_{\epsilon}$.

Moreover, there exist some positive $\tau_{0}$ and an $\epsilon_5 < \epsilon_4$ such that $|y(\tau_{0})-h_{\epsilon}(x(\tau_{0}))| < d$ for all $\epsilon \in (0,\epsilon_5)$, whenever $(x(\tau),y(\tau))$ is the solution to \eqref{Fastequ} with $(x(0),y(0))=(x_{0},y_{0})\in D_{\epsilon}$. This implies that $(x(\tau_{0}),y(\tau_{0})) \in N_{\epsilon}$. The corresponding proof is similar as Lemma 8 in \cite{WS08} with the positive invariance of $D_{\epsilon}$
(In fact, $|y(\tau)-h_{\epsilon}(x(\tau))| \leq |y(\tau)-h_{0}(x(\tau))| + |h_{0}(x(\tau)) - h_{\epsilon}(x(\tau))|$,
the assumption $A5$ and Lemma \ref{singularlemm} $1)$ imply that there exist a $\tau_{0}>0$ and  $\epsilon_5 < \epsilon_4$ such that $|y(\tau_{0})-h_{0}(x(\tau_{0}))|<\dfrac{d}{2}$ and $|h_{0}(x(\tau_{0})) - h_{\epsilon}(x(\tau_{0}))|<\dfrac{d}{2}$).
Let $\phi^{\epsilon}_{\tau}$ be the flow of \eqref{Fastequ},
and $F(D_{\epsilon}) = \{ \phi^{\epsilon}_{\tau_0}((x,y)):(x,y)\in D_{\epsilon} \} \subset N_{\epsilon}$. Since $\phi^{\epsilon}_{\tau_0}$ is diffeomorphism, the set $G_{\epsilon}=F(\text{Int } D_{\epsilon})\cap T_{\epsilon, d}(M_{\epsilon},L_d)$ is open and dense in $F(D_{\epsilon})$. So, $F^{-1}(G_{\epsilon})$ is open and dense in $D_{\epsilon}$. Let $\mathcal{P}_{\epsilon}=F^{-1}(G_{\epsilon})$. For $(x,y)\in \mathcal{P}_{\epsilon}$, the fact that $\phi^{\epsilon}_{\tau_0}((x,y))\in T_{\epsilon, d}(M_{\epsilon},L_d)$ implies the $\omega$-limit set $\omega((x,y))$ containing no equilibrium is a single closed orbit by the asymptotic phase property.

We have completed the proof of Theorem \ref{thm} by taking $\epsilon^{*} = \min \{\epsilon_3,\epsilon_5 \}$.
\end{proof}

\section{An Example}

In this section, we consider the following singular perturbed ODE's system:
\begin{equation}\label{singular-high-cone}
\begin{cases}
\frac{dx}{dt}  = x-y-\frac{3}{2}xz^2-\frac{x(x^2+y^2)}{2}+\epsilon xw,\\
\frac{dy}{dt}  = x+y-\frac{3}{2}yz^2-\frac{y(x^2+y^2)}{2}+\epsilon yw,\\
\frac{dz}{dt}  = -z-\frac{z^3}{2}-\frac{3}{2}z(x^2+y^2)+\epsilon zw,\\
\epsilon \frac{dw}{dt}  = -w+x+y+z,
\end{cases}
\end{equation}
where the parameter $\epsilon< \epsilon_{0}=0.1$. Let
\begin{equation*}
\begin{split}
D_{\epsilon} = D = \{ & (x,y,z,w) : |x| \leq 4, |y| \leq 4, |z| \leq 4,|w| \leq 16 \}.
\end{split}
\end{equation*}
Assume that $U\subset \mathbb{R}^{3}$, $V\subset \mathbb{R}^{1}$ are open bounded sets such that $D_{\epsilon}\subset U\times V$.

We consider the limiting system:
\begin{equation}\label{limit-singular-high-cone}
\begin{cases}
\frac{dx}{dt} = x-y-\frac{3}{2}xz^2-\frac{x(x^2+y^2)}{2},\\
\frac{dy}{dt} = x+y-\frac{3}{2}yz^2-\frac{y(x^2+y^2)}{2},\\
\frac{dz}{dt} = -z-\frac{z^3}{2}-\frac{3}{2}z(x^2+y^2).
\end{cases}
\end{equation}
This system is inspired by Ortega and S\'{a}nchez \cite{OS00}, where they found a stable closed orbit of so called $P$-competitive systems (see \cite[P2914]{OS00}) such that every orbit tends to the closed orbit or the origin.
It is easy to see that for the limiting system \eqref{limit-singular-high-cone}, the origin is the unique equilibrium point and the linearization in the origin has eigenvalues $-1$ and $1\pm i$.

Let $P$ be a diagonal matrix
\begin{equation*}
P=\begin{pmatrix} -1 & 0 & 0  \\ 0 & -1 & 0 \\ 0 & 0 & 1  \end{pmatrix}.
\end{equation*}
Let $\lambda:\mathbb{R}^{3} \to \mathbb{R}$ be a continuous function (not necessarily positive) such that the matrices
\begin{equation}\label{negative-matrix}
PDF(\xi)+DF(\xi)^{*}P+\lambda (\xi)P
\end{equation}
are negatively definite, for each $\xi\in U$, where $F$ is the vector field of \eqref{limit-singular-high-cone} and $DF(\xi)^{*}$ stands for the transpose of $DF(\xi)$.
In fact, a straightforward calculation yeilds that for any function $\lambda$ with $3x^2+3y^2+3z^2-2 < \lambda (x,y,z) < 3x^2+3y^2+3z^2+2$, the matrices \eqref{negative-matrix} are negatively definite.

Define
\begin{equation*}\label{convex-cone}
C=\{\xi \in \mathbb{R}^{3} : \langle P\xi,\xi \rangle \leq 0 \},
\end{equation*}
where $\langle \cdot ,\cdot \rangle$ is the inner product in $\mathbb{R}^{3}$.
It is clear that $C$ is a $2$-solid cone and system \eqref{limit-singular-high-cone} is $C$-cooperative (see \cite{Sanchez09} or \cite{FWW17,FWW19}).

One can check that the assumptions (A1)-(A6) hold for system \eqref{singular-high-cone}, 
since the fast variable $w$ has the simple case that the fourth equation of \eqref{singular-high-cone} is linear for $w$.
Thus, we are able to obtain the following generic Poincar\'{e}-Bendixson Theorem:
\begin{thm}\label{application}
There exists a positive constant $\epsilon^{*} < \epsilon_{0}$ such that for each $\epsilon \in (0,\epsilon^{*})$, system \eqref{singular-high-cone} has the following property: there exists an open and dense subset $\mathcal{P}_{\epsilon} \subset  D_{\epsilon}$ such that, for each $z\in \mathcal{P}_{\epsilon}$, $\omega$-limit set $\omega(z)$ that contains no equilibrium is a single closed orbit.
\end{thm}

By a numerical simulation with $\epsilon=0.05$, we illustrate that the flow $\phi$ of system \eqref{singular-high-cone} with initial value $(x_0,y_0,z_0,w_0)=(2,2,3,12)\in D_{\epsilon}$ tends to a periodic orbit in Figure 1. It is worth pointing out that this initial value is not specially selected, actually, the orbits initiated from generic points in $D_{\epsilon}$ will be attracted to closed orbits.
\vspace{\bigskipamount}
\par \hfill
\begin{minipage}{.99\linewidth}
  \centerline{\includegraphics[width=9cm]{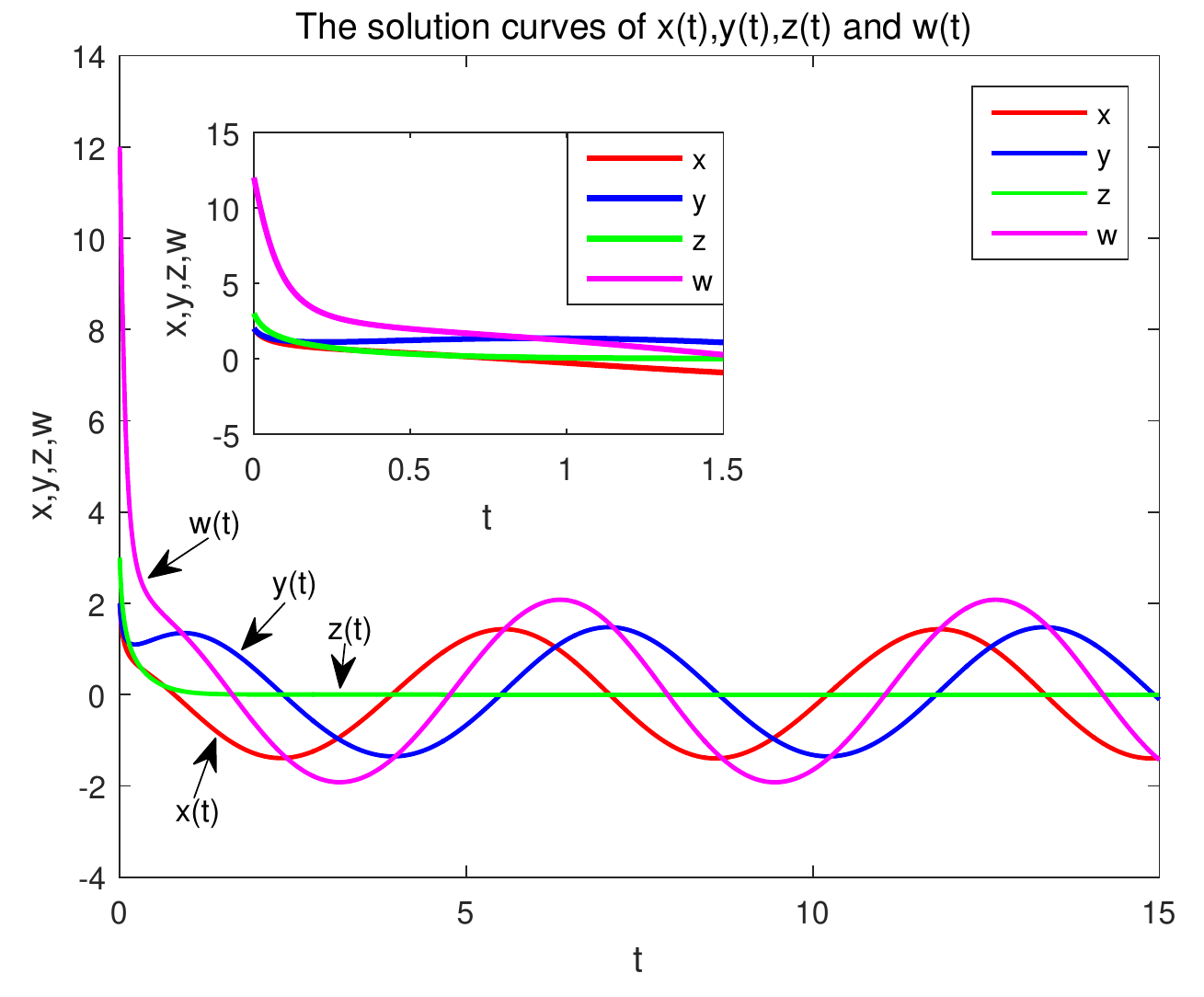}}
  \center{{\bf Fig 1}:\normalsize { The solution curves of system \eqref{singular-high-cone} with initial value $(x_0,y_0,z_0,w_0)$
  \par $=(2,2,3,12)$. The small graph is a detail view when t $\in [0,1.5]$.
  }}

\end{minipage}
\\[5pt]
\par The corresponding phase portraits are shown in Figure 2. They illustrate that the projections of the flow $\phi_t(x_0,y_0,z_0,w_0)$ onto $3$-dimensional spaces tend to periodic orbits, respectively. This implies that the flow $\phi_t(x_0,y_0,z_0,w_0)$ converges to a closed orbit in four dimensional space.
\par
\begin{figure}[!h]
	\centering
	\subfigure[\scriptsize The projection of the flow $\phi_t(x_0,y_0,z_0,w_0)$ in space O-xyz and the flow $\psi_t(x_0,y_0,z_0)$ of limiting system
 \eqref{limit-singular-high-cone}.]{
		\includegraphics[width=7.3cm]{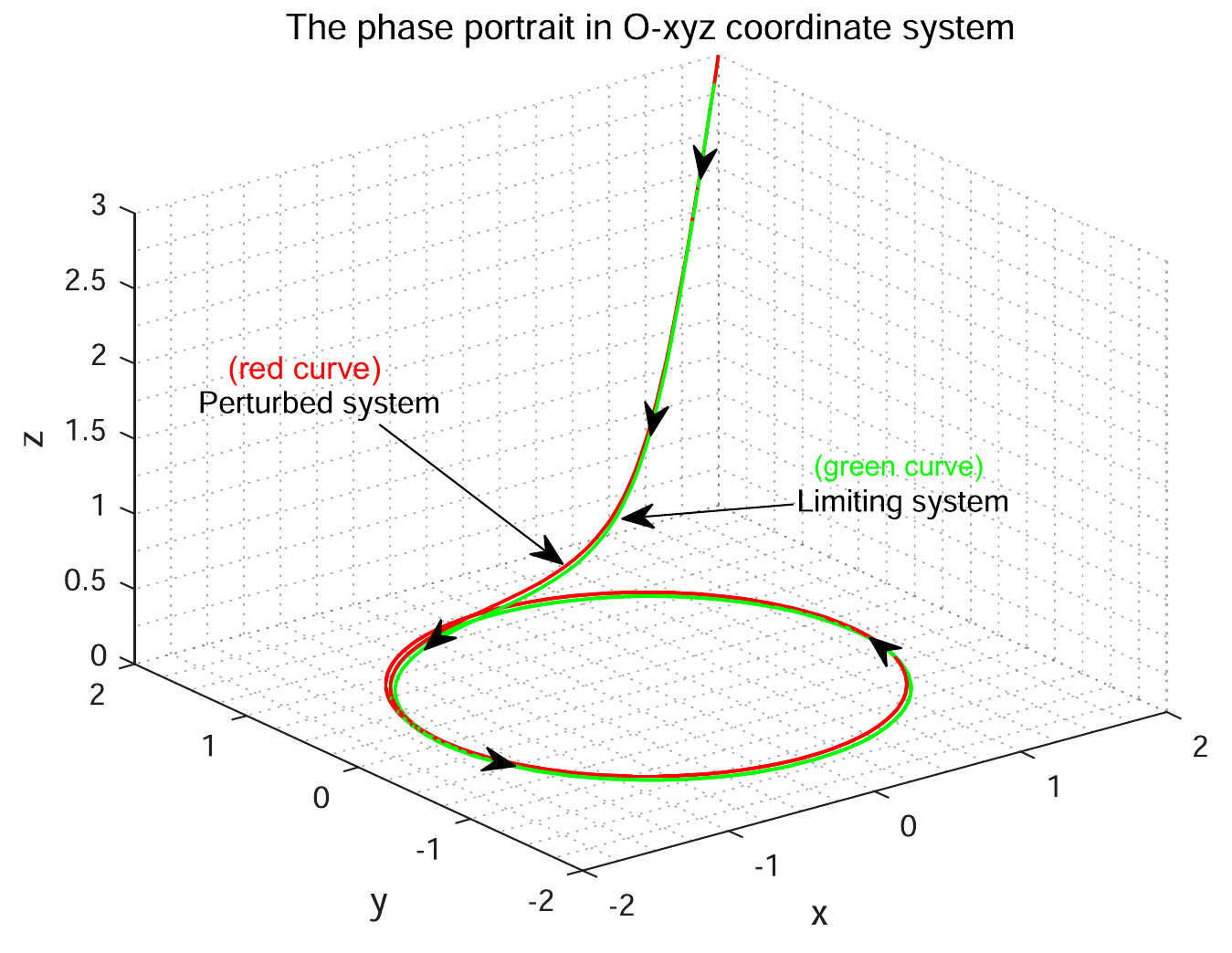}
		%\caption{fig1}
	}
	\quad
	\subfigure[\scriptsize The projection of the flow $\phi_t(x_0,y_0,z_0,w_0)$ in space O-xyw.]{
		\includegraphics[width=7.3cm]{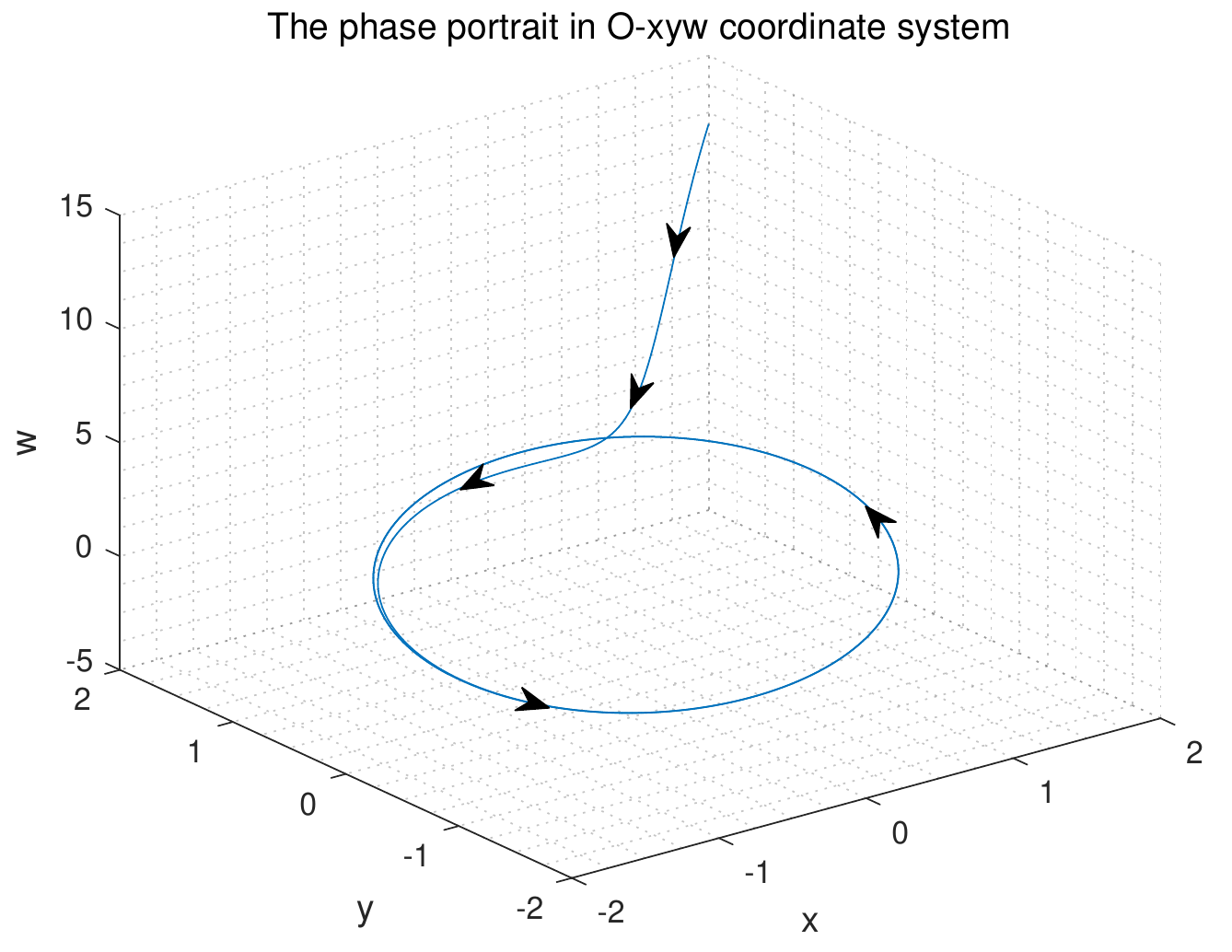}
	}
	\quad
	\subfigure[\scriptsize The projection of the flow $\phi_t(x_0,y_0,z_0,w_0)$ in space O-yzw.]{
		\includegraphics[width=7.3cm]{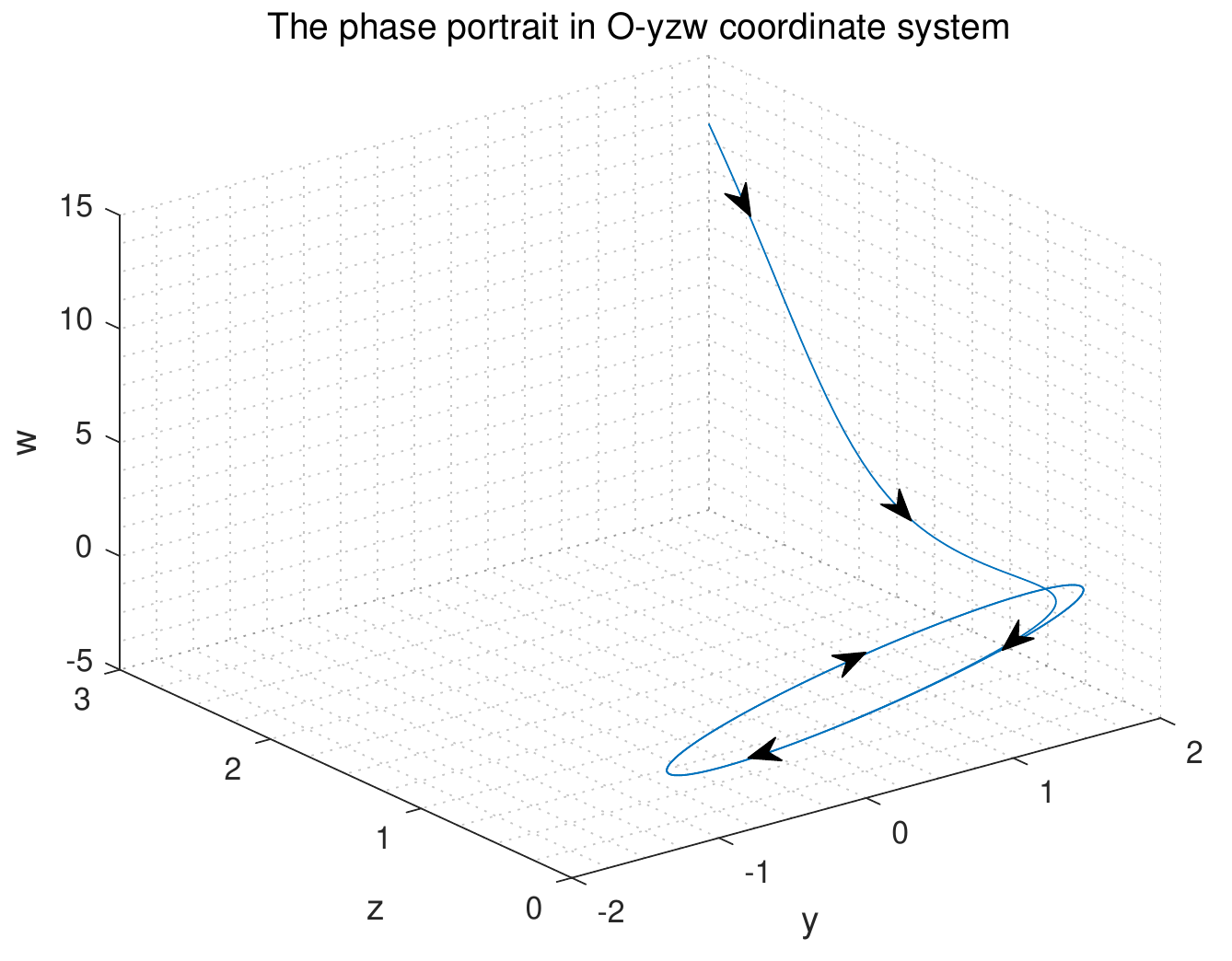}
	}
	\quad
	\subfigure[\scriptsize The projection of the flow $\phi_t(x_0,y_0,z_0,w_0)$ in space O-xzw.]{
		\includegraphics[width=7.3cm]{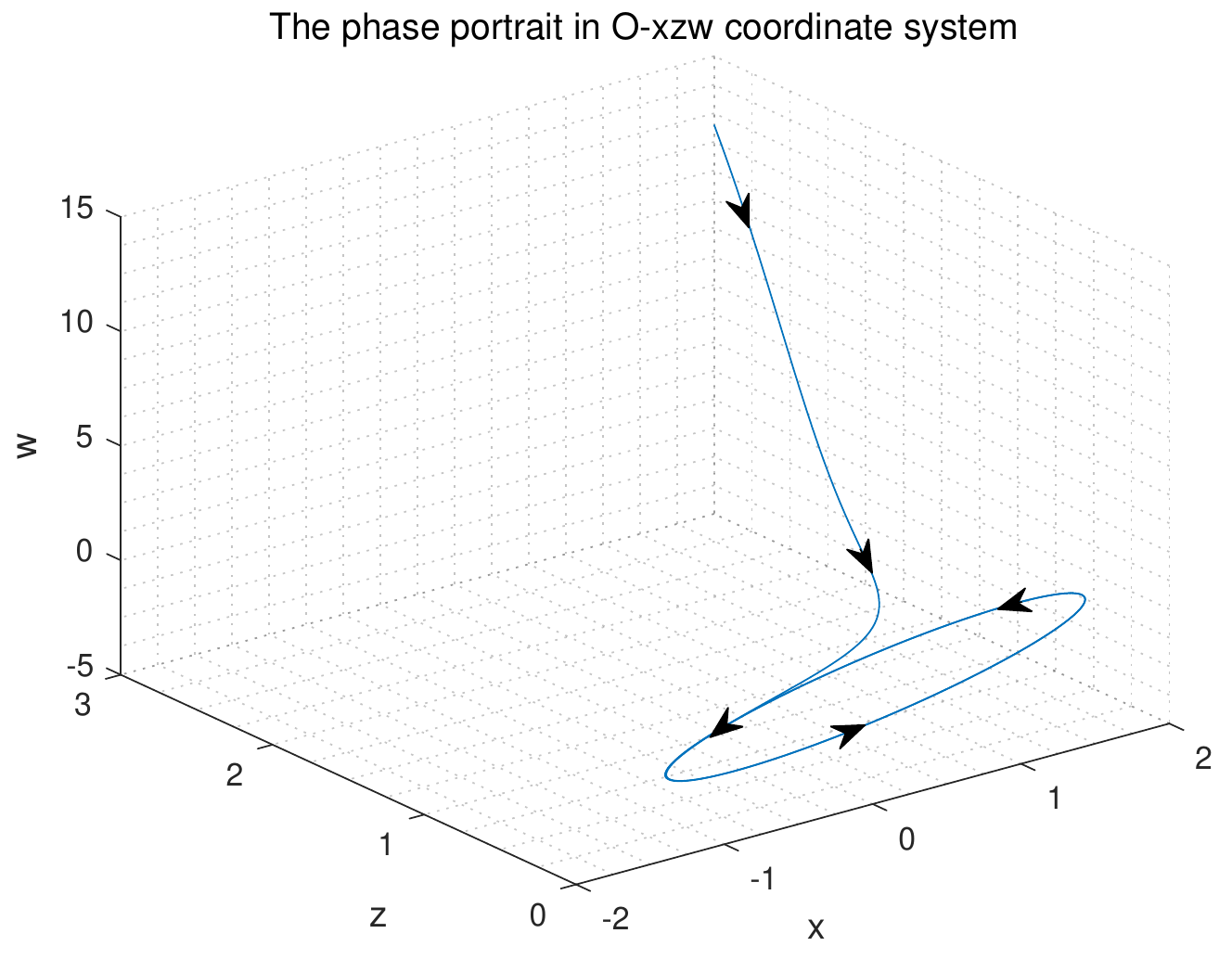}
	}
\center{{\bf Fig 2}: The corresponding phase portraits of {Fig 1}.}
\end{figure}

\begin{rmk}
Define the set $K=\{\xi \in \mathbb{R}^{3} : \langle P\xi,\xi\rangle \geq 0, \langle \xi,v_+\rangle \geq 0 \}$, where $v_+$ is an eigenvector of $P$ with respect to the positive eigenvalue $1$. It is clear that $K$ is a usual convex cone with nonempty interior. Moreover, the set $C=\{\xi \in \mathbb{R}^{3} : \langle P\xi,\xi\rangle \leq 0 \}$ can be written as $C = \overline{\mathbb{R}^{3} \backslash(K \cup (-K))}$. We say a flow $\psi_t$ is {\it strongly competitive} with respect to the partial ordering induced by $K$ if $y-x\in \text{Int }K$ whenever $\psi_{t}(y)-\psi_{t}(x)\in K\backslash\{0\}$ and $t>0$.

We emphasize that {\it a flow $\psi_t$ is strongly monotone with respect to $C$ if and only if $\psi_t$ is strongly competitive; and moreover, $D\psi_{t}(x)[C\backslash \{0 \}]\subset \text{Int }C$ for $x\in U$ and $t>0$ if and only if $D\psi_{-t}(\psi_{s}(x))[K\backslash \{0 \}]\subset \text{Int }K$ for $x\in U$, $0< t\leq s$}.

As a consequence, the fact that system \eqref{limit-singular-high-cone} is $C$-cooperative implies that its flow $\psi_{t}^{0}$ satisfies that $D\psi_{-t}^{0}(\xi)[K\backslash \{0 \}]\subset \text{Int }K$ for $t>0$ and $-t\in I(\xi)$, where $I(\xi)\subset \mathbb{R}$ denotes the maximal interval of existence of the solution passing though $\xi$. By our previous work in \cite{NW18,Niu19}, there exists a positive constant $\epsilon_{*} < \epsilon_{0}$ such that for each $\epsilon \in (0,\epsilon_{*})$, the $\omega$-limit set of \eqref{singular-high-cone} is equivalent to an $\omega$-limit set of an eventually competitive system.
Moreover, if the orbit passing though any point $\zeta=(x,y,z,w)\in D_{\epsilon}$ is tracked by a full orbit in $K_{\epsilon}$, then the $\omega$-limit set $\omega(\zeta)$ containing no equilibrium is a single closed orbit. The main result in the present paper provides a different view that a flow which is competitive can be viewed as a monotone system with respect to a high-rank cone, by which we can ignore the priori drawback restriction of the so called Non-oscillation Principle for the full orbits in competitive systems.

\end{rmk}

\section*{Acknowledgments}
The authors are greatly indebted to Professor Yi Wang for valuable suggestions which led to much improvement of this paper.


\begin{thebibliography}{10}

%%%%%%%%%%%%%%%%%%%%%%%%%%%%%%%%%%%%%%%%%%%%%%
%\bibitem{Chow1991}
%S.-N. Chow, X. Lin and K. Lu, \rm{Smooth invariant foliations in infinite-dimensional spaces}, J. Diff. Eqns. \textbf{94} %(1991), 266-291.

%\bibitem{hale1988asymptotic}
%J. Hale, \rm{Asymptotic behavior of dissipative systems}, Mathematical
%surveys and monographs, Vol.~25, Amer. Math. Soc. 1988.

%\bibitem{SWZ}
%W. Shen, Y. Wang and D. Zhou, \rm{Almost automorphically and almost periodically forced circle flows of almost periodic %parabolic equations on $S^1$}, 2015, arXiv:1507.01709.


%%%%%%%%%%%%%%%%%%%%%%%%%%%%%%%%%%%%%%%%%%%%%%




\bibitem{Fenichel79}
N. Fenichel, \rm{Geometric singular perturbation theory for ordinary differential equations},
J. Diff. Equ., \textbf{31}(1979), 53-98.



\bibitem{FO91}
G. Fusco and W. Oliva, \rm{A Perron theorem for the existence of invariant subspaces},
Ann. Mat. Pura. Appl., \textbf{160}(1991), 63-76.



\bibitem{FWW17}
L. Feng, Y. Wang and J. Wu, \rm{Semiflows ``monotone with respect to highrank cones" on a Banach space},
SIAM J. Math. Anal., \textbf{49}(2017), 142-161.


\bibitem{FWW19}
L. Feng, Y. Wang and J. Wu, \rm{Generical behavior of flows strongly monotone with respect to high-rank cones},
J. Differential Equations, \textbf{275} (2021), 858--881.



\bibitem{HessandPolacik}
P. Hess and P. Pol\'{a}\v{c}ik, \rm{Boundeness of prime periods of stable cycles and convergence to fixed points in discrete monotone dynamical systems},
SIAM J. Math. Anal. \textbf{24}(1993), 1312-1330.




\bibitem{H82}
M. W. Hirsch, \rm{Systems of differential equations which are competitive or cooperative I: limit sets},
SIAM J. Appl. Math., \textbf{13} (1982), 167--179.


\bibitem{H85}
M. W. Hirsch, \rm{Systems of differential equations which are competitive or cooperative II: convergence almost everywhere},
SIAM J. Math. Anal., \textbf{16}(1985), 423-439.


\bibitem{H88}
M. W. Hirsch, \rm{Systems of differential equations which are competitive or cooperative III: compting species},
Nonlinearity, \textbf{1} (1988), 51--71.



\bibitem{H90}
M. W. Hirsch, \rm{Systems of differential equations which are competitive or cooperative IV: Structural stability in three dimensional systems},
SIAM J. Math. Anal., \textbf{21} (1990), 1225--1234.


\bibitem{H89}
M. W. Hirsch, \rm{Systems of differential equations that are competitive or cooperative V: convergence in $3$-dimensional systems},
J. Differential Equations, \textbf{80} (1989), 94--106.


\bibitem{H91}
M. W. Hirsch, \rm{Systems of differential equations that are competitive or cooperative VI: a local $C^r$ closing lemma for $3$-dimensional systems},
Ergodic Theory Dynam. Systems, \textbf{11} (1991), 443--454.




\bibitem{HS03}
M. W. Hirsch and H. L. Smith, \rm{Monotone Systems, A Mini-review, Proceedings of the First
Multidisciplinary Symposium on Positive Systems (POSTA 2003)}, Luca Benvenuti, Alberto
De Santis and Lorenzo Farina (Eds.) Lecture Notes on Control and Information Sciences Vol.
294, Springer-Verlag, Heidelberg, 2003.

\bibitem{HS05}
M. W. Hirsch and H. L. Smith, \rm{Monotone dynamical systems},
Handbook of Differential Equations: Ordinary Differential Equations, Vol.~2, Elsevier, Amsterdam 2005.



\bibitem{Jones95}
C.K.R.T. Jones, \rm{Geometric singular perturbation theory},
Dynamical Systems (Montecatini Terme,1994). Lect. Notes in Math., 1609, Springer, Berlin, 1995.



\bibitem{KLS89}
M. A. Krasnosel'skii, J. A. Lifshits and A. V. Sobolev, \rm{Positive Linear Systems, the Method of Positive Operators},
Heldermann Verlag, Berlin, 1989.



\bibitem{Nipp92}
K. Nipp, \rm{Smooth attractive invariant manifolds of singularly perturbed ODE's},
Research Report, (1992), 92--13.



\bibitem{Niu19}
Lin Niu, \rm{Eventually competitive systems generated by perturbations},
Electronic Journal of Differential Equations, \textbf{121}(2019), 1-12.


\bibitem{NW18}
Lin Niu and Yi Wang, \rm{Non-oscillation principle for eventually competitive and cooperative systems},
Discrete Contin. Dyn. Syst. B, \textbf{24} (2019), 6481-6494.



\bibitem{OS00}
R. Ortega and L. S\'{a}nchez, \rm{Abstract competitive systems and orbital stability in $\mathbb{R}^{3}$},
Proc. Amer. Math. Soc., \textbf{128}(2000), 2911-2919.



\bibitem{Polacik89}
P. Pol\'{a}\v{c}ik, \rm{Convergence in smooth strongly monotone ows defined by semilinear parabolic equations},
J. Differential Equations, \textbf{79}(1989), 89-100.

\bibitem{Polacik88}
P. Pol\'{a}\v{c}ik, \rm{Generic properties of strongly monotone semiflows defined by ordinary and parabolic differential equations},
Qualitative theory of differential equations (Szeged 1988), 519-530, Colloq. Math. Soc. J\'{a}nos Bolyai, 53, North-Holland, Amsterdam, 1990.

\bibitem{PolacikandTerescak}
P. Pol\'{a}\v{c}ik and I. Tere\v{s}\v{c}\'{a}k, \rm{Convergence to cycles as a typical asymptotic behavior in smooth strongly monotone discrete-time dynamical systems},
Arch. Ration. Mech. Anal. \textbf{116}(1992), 339-360.

\bibitem{Saka90}
K. Sakamoto, \rm{Invariant manifolds in singular perturbation problems for ordinary differential equations},
Proc. R. Soc. Edinb. A, \textbf{116}(1990), 45-78.




\bibitem{Sanchez09}
L. A. S\'{a}nchez, \rm{Cones of rank $2$ and the Poincar\'{e}-Bendixson property for a new class of monotone systems},
J. Differential Equations, \textbf{216}(2009), 1170-1190.



\bibitem{S95}
H. L. Smith, \rm{Monotone Dynamical Systems, an introduction to the theory of competitive and cooperative systems},
Math. Surveys and Monographs, 41, Amer. Math. Soc., Providence, Rhode Island 1995.



\bibitem{S17}
H. L. Smith, \rm{Monotone dynamical systems: Reflections on new advances and applications}, Discrete Contin. Dyn. Syst., \textbf{37}(2017), 485-504.



\bibitem{SmithandThieme}
H. Smith and H. Thieme, \rm{Quasi convergence and stability for strongly order-preserving semiflows},
SIAM J. Math. Anal. \textbf{21}(1990), 673-692.



\bibitem{RASmith79}
R. A. Smith, \rm{The Poincar\'{e}-Bendixson theorem for certain differential equations of higher order},
Proc. of Royal Soc. of Edinburgh, \textbf{83A}(1979), 63-79.

\bibitem{RASmith80}
R. A. Smith, \rm{Existence of periodic orbits of autonomous ordinary differential equations},
Proc. of Royal Soc. of Edinburgh A, \textbf{85}(1980), 153-172.


\bibitem{RASmith87}
R. A. Smith, \rm{Orbital stability for ordinary differential equations},
J. Differential Equations, \textbf{69}(1987), 265-287.



\bibitem{Terescak94}
I. Tere\v{s}\v{c}\'{a}k, \rm{Dynamics of $C^1$ smooth strongly monotone discrete-time dynamical systems},
preprint, Comenius University, Bratislava, 1994.


\bibitem{WS08}
L. Wang and E. D. Sontag, \rm{Singularly perturbed monotone systems and an application to double phosphorylation cycles},
J. Nonlinear Sci., \textbf{18}(2008), 527-550.


\bibitem{WangandYao}
Y. Wang and J. Yao, \rm{Dynamics alternatives and generic convergence for $C^1$-smooth strongly monotone discrete dynamical systems},
J. Differential Equations, \textbf{269}(2020), 9804-9818.


\bibitem{WeissandMargaliot}
E. Weiss and M. Margaliot, \rm{A generalization of linear positive systems with applications to nonlinear systems: Invariant sets and the Poincar\'{e}-Bendixson property},
Automatica J.IFAC, \textbf{123}, 2021.





	
\end{thebibliography}
\end{document}